%% file: main.tex
\documentclass{article}

\usepackage[english]{babel}
\usepackage[usenames,dvipsnames,svgnames,table]{xcolor}
\usepackage{amsfonts}
\usepackage{enumerate}
\usepackage{amsthm,amsmath,amssymb}
\usepackage{cite}
\usepackage[normalem]{ulem}
\usepackage{mdwlist,enumitem}
\usepackage{authblk}

\input{commands}     										 %
\title%[Transfer principles for integrability and boundedness]
{Integration and Cell Decomposition in P-minimal Structures}

\author[1]{Pablo Cubides Kovacsics
\thanks{The research leading to these results has received funding from the European Research Council, ERC Grant nr. 615722, MOTMELSUM, 2014 - 2019. }}
\author[2]{Eva Leenknegt
\thanks{During the realization of this project, the second author was a postdoctoral fellow of the Fund for Scientific Research - Flanders (Belgium) (F.W.O.).}
}
\affil[1]{Laboratoire Paul Painlev\'e, Universit\'e de Lille 1, CNRS U.M.R. 8524, 59655 Villeneuve d'Ascq Cedex, France.}
\affil[2]{Department of Mathematics, KULeuven, Celestijnenlaan 200B, 3001 Heverlee, Belgium.}
\begin{document}

\maketitle
\begin{abstract}
We show that the class of $\Lm$-constructible functions is closed under integration for any $P$-minimal expansion of a $p$-adic field $(K,\Lm)$. This generalizes results previously known for semi-algebraic and sub-analytic structures. 
As part of the proof, we obtain a weak version of cell decomposition and function preparation for  $P$-minimal structures, a result which is independent of the existence of Skolem functions.
%The result is obtained from weak versions of cell decomposition and function preparation which we prove for general $P$-minimal structures. 
A direct corollary is that Denef's results on the rationality of Poincar\'e series hold in any $P$-minimal expansion of a $p$-adic field $(K,\Lm)$.
\end{abstract}

%\tableofcontents

\input{introduction}
\input{celldecomposition}

%\input{accumulated_junk}
\input{preparation}
%\input{relativePminimality}
\input{integration}

\appendix
\input{wellordered}

\bibliographystyle{amsplain}
\bibliography{../Bibliografie}

\end{document}

%% file: commands.tex
\newtheorem{theorem}{Theorem}[section]
\newtheorem*{theorem*}{Theorem}
\newtheorem{thm}[theorem]{Theorem}

\newtheorem*{prop*}{Proposition}
\newtheorem{prop}[theorem]{Proposition}
\newtheorem{corollary}[theorem]{Corollary}
\newtheorem{cor}[theorem]{Corollary}

\newtheorem{claim}[theorem]{Claim}
\newtheorem{lem}[theorem]{Lemma}
\newtheorem{conjecture}{Conjecture}
\newtheorem{def-theorem}[theorem]{Theorem-Definition}
\newtheorem*{statement*}{Statement}

\theoremstyle{definition}
\newtheorem*{remark}{Remark}
\newtheorem*{thm*}{Theorem}

\newtheorem{defn}[theorem]{Definition}

\newcommand{\Q}{\mathbb{Q}}
\newcommand{\QQ}{\Q}

\renewcommand{\AA}{\mathbb{A}}
\newcommand{\Z}{\mathbb{Z}}
\newcommand{\ZZ}{\Z}
\newcommand{\N}{\mathbb{N}}
\newcommand{\NN}{\mathbb{N}}

\newcommand{\R}{\mathbb{R}}
\newcommand{\RR}{\R}

\newcommand{\Lm}{\mathcal{L}}
\newcommand{\cL}{\Lm}

\newcommand{\cO}{\mathcal{O}}

\newcommand{\ord}{\mathrm{ord}}

\newcommand{\Lring}{\Lm_{\text{ring}}}
\newcommand{\Lringtwo}{\Lm_{\text{ring,2}}}
\newcommand{\Lpres}{\Lm_{\text{Pres}}}

\newcommand{\ac}{\text{ac}\,}
\newcommand{\acm}{\mathrm{ac}_{m}\,}

%% file: introduction.tex
% !TEX root = main.tex
\section{Introduction}

One of the main results of this paper is that, for arbitrary $P$-minimal structures over $p$-adic fields, the class of constructible functions is closed under integration. This generalizes a result which was previously known only for semi-algebraic and sub-analytic structures (and some intermediary cases). 

As part of the proof, we obtain the second main result of this paper: a version of cell decomposition and function preparation for $P$-minimal structures. While our version is somewhat weaker than what was obtained in previous attempts by e.g. Mourgues \cite{mou-09}, it does not depend on the existence of definable Skolem functions. %Moreover, they  are still strong enough to obtain results concerning p-adic integration, in fact they form the core of the proof of our first main result. %The second main result of this paper 

In this introduction we give an informal motivation of our approach, discussing the historical connections between $p$-adic integration, rationality of Poincar\'e series, and cell decomposition. We will present exact statements of our main results in the next section.  
%Specifically for $p$-adic semi-algebraic and sub analytic structures, a lot is known about integration and cell decomposition. This paper is an attempt to generalize some of these results to all $P$-minimal structures. We will first give a short historic motivation, explaining which are the motivating problems, and what needs to be generalized. This first section will deliberately be kept very informal, the exact statements of our results will be deferred to the next section.
\\\\
The study of $p$-adic integrals was motivated by a number-theoretic question. It was conjectured by Borevich and Shafarevich that for $f(x) \in \ZZ[x]$,  Poincar\'e series like e.g. $P(T):= \sum_{m\in \N}N_mT^m$, where 
\[N_m := \#\{x \in (\ZZ/p^m\ZZ)^n\mid f(x) \equiv 0 \mod p^m\},\]
 are rational functions of $T$. This was originally proven by Igusa \cite{igusa-74, igusa-75, igusa-78}. Later, Denef \cite{denef-84} {obtained a similar, more general result. He gave two proofs that were based on Macintyre's quantifier elimination for semi-algebraic sets \cite{mac-76}, one using resolution of singularities, and one where he introduced  cell decomposition techniques.   We refer to \cite{denef-2000} for a comparison of both approaches.  %for which he used and developed a $p$-adic cell decomposition theorem. 

A first step towards a proof is %connecting the series $P(T)$ to a $p$-adic integral. 
%{\color{green}For instance, in the example given above, one has that for $s \in \R, s>0$,  $P(p^{-(m+1)}p^{-s} = \frac{p}{p-1} I(s)$, where
%\[I(s) = \int_D|w|^s|dx||dw|,\]
%where $D$ is some semi-algebraic subset of $\Z_p^{m+1}$}
the realization that the terms of a Poincar\'e series can be connected to the measure of certain semi-algebraic sets, and hence to $p$-adic integrals. For instance, one can check that
\[ N_m = p^{nm} \cdot \mu(\{x\in \Z_p^n \mid \ord f(x) \geqslant m\}),\]
where $\mu$ is the Haar measure normalized such that $\mu(\Z_p) = 1$, and $\ord$ denotes the valuation map $\ord: \Q_p \to \Z\cup \{\infty\}$. 
To prove rationality, one needs to understand how the measure of this family of definable sets depends on $m$. (If the dependence is tame enough, identities like $\sum _{n=0}^{\infty}{x^n} = \frac1{1-x}$ can be used to deduce rationality.)
Hence the focus shifts to the computation of (families of) $p$-adic integrals.%: do these depend on the parameters in a definable way? 

In a $p$-adic integral, the integrand is usually the valuation `$\ord f$' or the $p$-adic norm $|f| := p^{-\ord f}$ of a semi-algebraic function $f$. More generally, one can consider so-called \emph{constructible} functions, i.e. $\Q$-linear combinations of definable functions $\alpha: \Q_p^r \to \Z$ and their induced functions $p^{\alpha}$.
 Note that it is natural here to work in a two-sorted structure $(\Q_p, \Z)$, adding the value group $\Z \cup \{\infty\}$ as a separate sort. The word definable should also be interpreted in that context, using a two-sorted language $\Lringtwo =(\Lring, \Lpres, \ord)$, consisting of the ring language $\Lring$ for the main sort $\Q_p$, the Presburger language $\Lpres = (+,-,<,\equiv_n)$ for the value group sort $\Z \cup \{\infty\}$, and the valuation map $\ord: \Q_p \to \Z \cup \{\infty\}$ to connect the sorts. 
 
 Denef showed in \cite{denef-85} that given a semi-algebraic function $f$,  the integral of $|f|$ equals a constructible function, a result which was later generalized to the sub-analytic setting by Cluckers, Gordon and Halupczok \cite{Clu-Gor-Hal-14}. We will discuss this in more detail in the next section. In section \ref{sec:integration}, we prove that this closure property holds in arbitrary $P$-minimal structures over $p$-adic fields. Corollary \ref{thm:rationalitygeneral} then provides an example of how the rationality of Poincar\'e series can be deduced from this.
\\\\
Let us now discuss the connection with cell decomposition techniques. The general philosophy is to partition a definable set $X$ in somewhat uniform parts, called cells. If $X$ is the domain of a function $f$, an additional goal may be to \emph{prepare} the function, i.e. choosing the partition in such a way that the function $f$, when restricted to each of the cells, has some additional nice properties. %For instance, in $o$-minimality one can assure that, up to a finite partitioning of the domain, a definable function is continuous, and either monotone or constant.
Cells are generally taken to be sets of the form %typically given by some formula $\phi(f_1(x), \ldots, f_r(x),t)$ where the general form $\phi$ is always the same. A $\phi$-cell would then be a set
\[\left\{(x,t) \in D \times T \left| \ \ \begin{array}{l} \text{a condition of a fixed form describing } \\ t \text{ in terms of the other variables } x \end{array}  \right\}\right.,\]
where $D$ is a definable set and $T$ is one of the sorts. For instance, when working with semi-algebraic sets, this fixed form is a formula stating that $x$ belongs to an interval-like set: \[ \ord \,a_1(x) \ \square_1  \ \ord (t-c(x)) \ \square_2 \ \ord \,a_2(x) \ \ \ \text{and} \ \ \ t-c(x) \in \lambda P_n,\]
where $\square_i$ may denote $<$ or \emph{no condition}, $\lambda \in K$ and the functions $a_i(x)$ are definable functions $D \to K$. We use the symbol $P_n$ for the set of (non-zero) $n$-th powers. Note that if $\lambda =0$, what we get is just the graph of the function $c(x)$.

What Denef \cite{denef-84, denef-86} showed is that given a semi-algebraic function $f:X\subseteq K^{n+1}\to K$, $X$ can be decomposed into cells such that on each cell $C$, $f_{|C}$ satisfies the following condition:

\[|f(x,t)|=|\lambda(t-c(x))|^{\frac{e}{n}}|h(x)|,\] 
where $h(x)$ and $c(x)$ are definable functions and $e,n$ are integer numbers. The version stated here is a reformulation by Cluckers \cite{clu-2000}, who  also obtained an analogous result  for subanalytic functions in \cite{clu-2003}.
%{\color{blue} What is exact contribution of Raf in semi-algebraic case? }
%Without going into details about the exact notion of cells used,  Denef's preparation theorem was a theorem of the style
%\begin{statement*}
%Let $f:X\subseteq \Q_p^{n+1} \to \Q_p$ be a semi-algebraic function. There exists a finite partition of $X$ into cells $C$, such that on each part $C$, $\ord f(x,t)$ depends on $t$ in an explicit way (like $\ord f(x,t) = g(x, \or(t-c(x))$)
%\end{statement*} 

This preparation result is particularly useful for integration purposes. Indeed, if the domain of a function $f: X \to K$ can be partitioned into a finite number of sets $\{ (x,t) \in D_i\times T \mid \phi_i(x,t)\}$ on each of which $f$ has the form prescribed above, then one gets that
\[ \int_{X} |f(x,t)| |dx||dt| =  \sum_{i} \int_{D_i} |h(x)|\left[\int_{\{t\in T\mid \phi_i(x,t)\}} |\lambda(t-c(x))|^{\frac{e}{n}}|dt|\right]|dx|.\]
Iteration of this theorem allows one to give very accurate descriptions of the value of the integral of $f$ and its dependence on possible parameters.
Similar strategies were applied for the subanalytic case, see e.g. \cite{denef-vdd-88, clu-2003}.
%The results we describe here were all generalized to the sub analytic case, using quite similar statements  (Denef, van den Dries?, Cluckers)
\\\\
%\textcolor{OliveGreen}{\begin{itemize}
%\item  Haskell and Macpherson came up with the notion, but obviously it is only going to be a useful concept if a sufficient number of interesting properties of its main examples (semi-algebraic and sub analytic sets), can be generalized to this context. 
%One natural question here is whether every such structure admits a cell decomposition, possibly together with a preparation theorem for definable functions. Results so far have been partial or conditional, like Mourgues cell decomposition. No preparation theorem is known.
%\item 
%\item In this paper we take a different approach. We decided to go back to the original motivation behind the above-mentioned results, which was the study of p-adic integrals. In this paper, we present a cell decomposition theorem for $P$-minimal structures (which is weaker than Mourgues, but unconditional), and a form of function preparation. While these theorems are less strong than their semi-algebraic counterparts, we show that they are sufficiently strong, by using them to generalize \textcolor{blue}{results on p-adic integration MOTIVATION IS GIVEN IN HISTORIC PART} to the P-minimal context.
%\end{itemize}}
When Haskell and Macpherson developed the notion of $P$-minimality \cite{has-mac-97}, it was natural to ask how much of the above ideas could be generalized to that setting. 
One of the most notable results so far in this context is Mourgues' cell decomposition theorem \cite{mou-09}. She showed that in a $P$-minimal structure which admits definable Skolem functions, any definable set $A \subseteq K^{r+1}$ can be partitioned in cells of the form
\[\{(x,t) \subseteq S\times K\mid \ord\, a(x) \ \square_1 \ \ord(t-c(x)) \ \square_2\ \ord\, b(x);\ t-c(x) \in \lambda P_n\},\]
where $a,b,c: S\to K$ are definable functions, $\square_i$ are either $<$ or \textit{no condition}, $P_n$ denotes the set of $n$-th powers and $\lambda \in K$. Moreover, she showed that the existence of definable Skolem functions is a necessary condition for the existence of a decomposition using cells of this form. Note that is currently not known whether all $P$-minimal structures admit definable Skolem functions. Work by the second author on reducts of $p$-adically closed fields \cite{lee-2012.1} seems to indicate that this may not be the case, and hence some caution is warranted when making this assumption. %In light of this, we consider it more prudent to avoid including the assumption of definable Skolem functions.

One way to deal with this uncertainty is to replace $P$-minimality by a (possibly) more restrictive notion, explicitly adding the existence of Skolem function as a condition. An example of this approach is the recent attempt of Darni\`{e}re (see \cite{dar-14} for his preprint), who  suggests a notion of so-called $P$-optimal structures. 
%In this preprint, he obtains a function preparation theorem similar to the Denef-Cluckers result. 
Even with these stronger assumptions, it is still an open problem whether some version of the Denef-Cluckers preparation theorem holds for arbitrary $P$-minimal structures.
%Note that this is only a cell decomposition theorem, and does not provide any function preparation, as the Denef/Cluckers result does. Whether such a preparation theorem exists is still an open problem.
%
%This seeming \textcolor{blue}{lack of progress} can be dealt with in several ways. One possibility (Darniere, Mourgues ) is to replace $P$-minimality by a more restrictive notion, and use the stronger underlying conditions to obtain stronger results.
\\\\
In this paper, we take an alternative approach: the results we present here  do not rely on the existence of definable Skolem functions. Instead,  we decided to shift the emphasis back to possible applications. While our versions of cell decomposition/preparation are certainly weaker than versions known for individual structures, they are still strong enough to prove the theorems that initially motivated their development.  Theorem \ref{thm:integrals} is an illustration of this.

The rest of the paper is organized as follows. In section \ref{sec:celdec}, we explain our results in more detail.  The proof of the cell decomposition and preparation theorems  will be given in section \ref{sec:prep}.
%Our arguments will use a result on definable sets of Presburger arithmetic, which we include in an appendix about ordered structures. 
 In section \ref{sec:integration} we show that constructible functions form a class that is stable under integration. This result will be used to derive the rationality of Poincar\'e series. % Finally, in section \ref{sec:relativity} we discuss some generalities regarding the notion of relative $P$-minimality and its relation with classical $P$-minimality. 
%We conclude the paper with an appendix on definably well-ordered structures.{\color{magenta} MORE??}

Our arguments use the fact that on every model $M$ of Presburger arithmetic, there exists a definable total order $\lhd$ of $M$, such that every definable subset of $M$ has a $\lhd$-minimal element.
We present this result in an appendix, as part of a more general framework. As a corollary, we obtain that Presburger arithmetic has elimination of imaginaries (a result already proven in 
\cite{clu-presb03}).
\\\\
The authors would like to thank Raf Cluckers for for stimulating conversations during the preparation of this paper. %During the realization of this project, the second author was a postdoctoral fellow of the Fund for Scientific Research - Flanders (Belgium) (F.W.O.).
 

%% file: celldecomposition.tex
% !TEX root = main.tex
\section{Overview of main results} \label{sec:celdec}
In this section we state the main results of this paper. Proofs will be deferred to later sections.

We  first fix some notations. Let $K$ be a $p$-adically closed field (that is, elementarily equivalent to a $p$-adic field). We use the notation $q_K$ for the number of elements of the residue field $k_K$, $\cO_K$ for the valuation ring of $K$, and $\pi_K$ for a uniformizing element. Write $\acm: K \to (\cO_K/\pi_K^m\cO_K)^{\times} \cup\{0\}$ for the $m$-th angular component map, which can be defined as
\[\acm(x) := \left\{\begin{array}{ll} \tilde{x}\hspace{-5pt} \mod \pi_K^m & \text{if } 0 \neq x = \pi_K^{\ord x}\tilde{x},\\
0& \text{if } x=0.
\end{array}\right.\]
In every expansion of a $p$-adically closed field $K$, such angular component maps exist and can be defined in a unique way, as was shown in Lemma 1.3 of \cite{clu-lee-2011}.
For notational purposes, we will fix a definable set $S \subseteq K^{m_0} \times \Gamma_K^{m_0}$ which we call a \emph{parameter set}. Given a set $X\subseteq S\times K$ and $s\in S$, we write \[X_s:=\{x\in K\mid (s,x)\in X\}\] to denote the fiber over $s$. 
Analogously, for a definable function $f:X\rightarrow \Gamma_K$, we use the notation $f_s(\cdot)$ for the function $f(s,\cdot):X_s\rightarrow \Gamma_K$. Given two sets $A$ and $B$, we write $\Pi_{A}:A\times B\to A$ for the projection onto $A$, and $\Pi_B:A\times B\to B$ for the projection onto $B$. For a positive integer $n\geq 1$, $A^{\leq n}$ denotes $\bigcup_{i=1}^n A^i$. 
%We also want to recall the definition of a $P$-minimal structure. A structure $(K,\Lm)$ is said to be (one-sorted) $P$-minimal if $\Lm \supseteq \Lring$, and for every $(K',\Lm)$ elementarily equivalent to $(K,\Lm)$, the $\Lm$-definable subsets of $K'$ are $\Lring$-definable
\\\\
%As mentioned before, when studying  constructible functions it is natural to add a separate sort for the value group. 
We will work with a two-sorted version of $P$-minimality, where we consider both the field sort and the value group sort $\Gamma_K \cup \{\infty\}$ to be of equal importance. 
 Let $(K, \Gamma_K;\Lm_2)$ be a two-sorted structure, with language $\Lm_2 = (\Lm, \cL_{Pres}, \ord)$. Here $\cL$, the language for the $K$-sort, is assumed to be an expansion of the ring language $\Lring$. For the value group sort $\Gamma_K\cup\{+\infty\}$, we use the language of Presburger arithmetic $\cL_{Pres} = (+,-,<,\equiv_n)$. The sorts are connected through the valuation map $\ord: K\to \Gamma_K\cup\{+\infty\}$. If the language $\Lm_2$ is clear from the context, we will just write $(K, \Gamma_K)$. By a definable set we mean definable with parameters. 
\begin{defn}
A two-sorted structure $(K, \Gamma_K; \Lm_2)$ with $\Lm_2 = (\Lm, \Lpres, \ord)$ and $\Lring \subseteq \Lm$ is said to be \emph{$P$-minimal} if the underlying structure $(K, \Lm)$ is $P$-minimal, that is, for every $(K',\Lm)$ elementarily equivalent to $(K,\Lm)$, the $\Lm$-definable subsets of $K'$ are $\Lring$-definable.
\end{defn}
This definition is motivated by the following observation, which was based on Wagner's results on definable functions in one variable on certain ordered abelian groups \cite{point-wag-2000}.
\begin{thm}[Cluckers\cite{clu-presb03}, Lemma 2 and Theorem 6]\label{thm:semialgpres}
Let $(K, \Lm)$ be a $P$-minimal field.
\item For any $\Lm$-definable set $X \subseteq (K^{\times})^m$, the set
\[\ord(X):=\{(\ord\, x_1,  \ldots , \ord\, x_m) \in \Gamma_K^m\mid (x_1, \ldots, x_m) \in X\}\]
is $\Lm_{\text{Pres}}$-definable.
\item Let $S \subseteq \Gamma_K^m$ be a Presburger-definable set. Then the set
\[\ord^{-1}(S):= \{ (x_1, \ldots\ , x_m) \in (K^{\times})^m \mid \ord\, x \in S\}\]
is $\Lring$-definable.
\end{thm}
\noindent This theorem implies that, given a (mono-sorted) $P$-minimal structure $(K, \Lm)$, the valuation map $\ord: K \to \Gamma_K \cup \{\infty\}$ induces a two-sorted structure $(K, \Gamma_K)$ where every definable subset of $\Gamma_K^m$ is $\Lpres$-definable, and hence it is natural to take $\Lpres$  as the language for the value group sort.
%We will work with a two-sorted version of $P$-mimimality, which can defined using the following observation.
%\begin{lemma}
%Let $(K, \Lm)$ be a (mono-sorted) $P$-minimal structure. The valuation map $\ord: K \to \Gamma_K$ induces a two-sorted structure $(K, \Gamma_K)$ with language $\Lpres$ for the value group sort.
%\end{lemma}
%\begin{proof}
%refer to raf's observation
%\end{proof}
%We consider both sorts to be equally important, and this will also be reflected in our theorems:
 
Note that a two-sorted $P$-minimal structure $(K,\Gamma_K)$ cannot have definable Skolem functions since any definable section of $\ord$ contradicts the assumption of $P$-minimality. 
\\\\
%{\color{magenta} MAYBE SOME OF THIS SHOULD MOVE TO THE HISTORIC INTRODUCTION?}
%{\color{green} The idea behind the notion of a cell is that of a set %typically given by some formula $\phi(f_1(x), \ldots, f_r(x),t)$ where the general form $\phi$ is always the same. A $\phi$-cell would then be a set
%\[\left\{(x,t) \in D \times T \left| \ \ \begin{array}{l} \text{a condition of a fixed form describing } \\ t \text{ in terms of the other variables } x \end{array}  \right\}\right.,\]
%where $D$ is a definable set and $T$ is one of the sorts. Classical examples include semi-algebraic cells as introduced by Denef \cite{denef-86}, and Presburger cells as introduced by Cluckers \cite{clu-presb03}.} In our context, the fixed condition used will depend on the sort $T$.
We will now explain our notion of cells and cell decomposition.
 We distinguish the following two kinds of cells:

\begin{defn}[Cells]\label{def:cell} Let $(K,\Gamma_K)$ be an $\cL_2$-structure.  
\begin{itemize}[leftmargin=*]
\item A subset $C\subseteq S\times K$ is a \emph{$K$-cell} if it is of the form 
\[C = \left\{(s,t) \in D\times K \ \left| \ \begin{array}{l} \alpha(s)\ \square_1 \ \ord(t-c(s)) \ \square_2 \ \beta(s),\\ \ord(t-c(s)) \equiv k\mod n,\\ \acm(t-c(s)) = \xi \end{array} \right\}\right.,\]
where $D$, the \emph{base} of the cell, is a definable subset of $S$, $c$ is a definable function $c:D\to K$, $\alpha, \beta$ are $\Lm_2$-definable functions $D\to\Gamma_K$, $k, n, m \in \NN$, $\xi \in \acm(K)$ and the symbols $\square_i$ may denote $<$ or \emph{no condition}. %If the function $c$ is definable, we say that the cell has \emph{definable centers}. 
\\If $\acm(t-c(s)) =0$ (and hence $t=c(s)$), one should ignore the first two conditions. %that $\ord(t-c(s)) \equiv k \mod n$. 
\item A subset $B\subseteq S\times \Gamma_K$ is a \emph{$\Gamma$-cell} if it is of the form
\[B= \left\{(s,\gamma)\in D\times \Gamma_K \left|\begin{array}{l} \alpha(s)\ \square_1 \ \gamma \ \square_2 \ \beta(s), \\
\gamma \equiv k\mod n \end{array}\right\}\right.,\]
where $D$, the \emph{base} of the cell, is a definable subset of $S$, $\alpha, \beta$ are definable functions $D\to\Gamma_k$, $k, n\in \NN$ and the squares $\square_i$ may denote $<$ or \emph{no condition}.  
\end{itemize}
\end{defn}
\noindent To each cell, one can associate a tuple $(\square,k,n,m,\xi)$, respectively $(\square,k,n)$, where $\square = (\square_1, \square_2) \in \{\emptyset, <\}^{2}$, $\xi\in \acm(K)$, and $(k,n,m)$ is a triple of non-negative integers such that $k<n$. These tuples will be referred to as the \emph{type} of the respective cells. We denote by $P_K$ (resp. $P_\Gamma$) the set of all possible types of $K$-cells (resp. $\Gamma$-cells).
\\
%In the next section we will prove a preparation theorem both for definable subsets $X\subseteq S\times K$ and $X\subseteq S\times \Gamma_K$. In both cases, the theorem includes a cell decomposition of $X$. In the case where $X\subseteq S\times \Gamma_K$, the cell decomposition will have the same form as in the classical theorem for semi-algebraic and sub-analytic sets. However, in the case where $X\subseteq S\times K$, our cell decomposition differs from the classical one. Even if our ``cell decomposition'' seems to be weaker, it is strong enough to have interesting applications as we will show in sections \ref{sec:integration} and \ref{sec:rationality}. 
\\\
%{\color{red} We obtain cell decomposition results for definable sets both of the form $X\subseteq S\times K$ and $X\subseteq S\times \Gamma_K$. In the latter case, we provide moreover a cell decomposition of the domain of definable functions $f:X\subseteq S\times \Gamma_K\to\Gamma_K$, explicitely describing  on each cell how the value of $f$ depends on the last variable. Such result, usualy known as ``function preparation'', corresponds to the following proposition. (This stands for an alternative to the next paragraph)}
We obtain cell decomposition results for definable sets both of the form $X\subseteq S\times K$ and $X\subseteq S\times \Gamma_K$. Let us first consider the case where the last variable belongs to the $\Gamma_K$-sort.
When $X$ is a set $X \subseteq S \times \Gamma_K$, we obtain a partition into a finite union of $\Gamma$-cells. Moreover, for functions $f: X \subseteq S\times \Gamma_K \to \Gamma_K$, we describe explicitly how, on each cell in the decomposition of $X$, the value of $f$ depends on the last variable. %This is the result we referred to as \emph{function preparation} in the introduction:%``function preparation'', corresponds to the following proposition.% (which is assumed to be from $\Gamma_K$). 

\begin{prop}[Function preparation] \label{prop:prep}
Let $f: X \subseteq S\times \Gamma_K \to \Gamma_K$ be definable in a $P$-minimal structure $(K,\Gamma_K)$. There exists a finite partition of $X$ in $\Gamma$-cells $C$, such that on each cell $C$  with type $(\delta, k, n)$, the function $f$ has the form
\[f(x,\gamma) = a\left(\frac{\gamma-k}{n}\right) + \delta(x),\]
where $a \in \ZZ,  n,k \in \NN$ and $\delta$ is a definable function $S\to \Gamma_K$. 
\end{prop}

If $X$ is a subset of  $S \times K$, the statement of our cell decomposition result is more subtle. The main difference between classical cell decomposition and $K$-cell decomposition arises at the level of centers. In the classical definition, the centers appear as the images of definable functions from the parameter set $S$. Instead, a $K$-cell decomposition provides a partition of $X$ into sets $X_i$, which are essentially a finite union of cells, together with a definable $\Sigma_i$ containing all their possible tuples of centers.  
%This still allows us to see $X$ decomposed into cells: the difference being that their centers cannot be picked definably from the parameters. Since no centers are used for definable sets $X\subseteq S\times \Gamma_K$, we get more classical results in that case. 
This type of decomposition is sufficiently strong for the computation of integrals: the Haar measure is translation invariant, and hence the centers are not of great importance here. 
%\begin{def-theorem}[$K$-cell decomposition]\label{defthm:celdecB}
%Let $X \subseteq S \times K$ be a set definable in a $P$-minimal structure $(K, \Gamma_K)$. There exists a finite partition of $X$ into definable sets $X_i \subseteq S_i \times K$. Each set $X_i$
%can be expressed as
%\[X_i = \left\{(s,x) \in S_i \times K \ \left| \ \forall (c_1, \ldots c_{r_i}) \in (\Sigma_i)_s: \bigvee_{j=1}^{r_i} C_{\delta_{ij}}(x,c_j, \alpha_{ij}, \beta_{ij};s)\right\}\right.,\]
%where $r_i \in \N$, $\Sigma_i \subseteq S_i \times K^{r_i}$ is a definable set and $C_{\delta_{ij}}(x,y,\alpha_{ij},\beta_{ij},s)$ are $K$-cell conditions. The tuple $\{(X_i )_{i},(\Sigma_i)_{i}, (C_{\delta_{ij}})_{i,j}\}$,  will be called a $K$-cell decomposition of $X$. 
%\end{def-theorem}

%Note that we do not obtain preparation for functions $f: S\times K \to \Gamma_K$. It is not clear to us whether such a result can be obtained without introducing additional assumptions, like the existence of Skolem functions (as recently studied by Darniere), or by imposing minimality conditions on $K \times \Gamma_K$.}
%, as the decomposition is only definable up to a translation. We provide two different statements of the theorem. The following form is the version which is used to prove {\color{blue} our result on closure of integration:}

\begin{def-theorem}[$K$-cell decomposition]\label{thm:celdec-trans}
Let $(K,\Gamma_K)$ be a $P$-minimal structure, and $X \subseteq S \times K$ be a definable set. There exists a finite partition of $X$ in sets $X_i \subseteq S_i \times K$. On each part $X_i$, there is an integer $r$ and $r$ associated $K$-cells $C_j$ of the form
\[ C_j := \left\{(s,t) \in S_i \times K \left| \begin{array}{l}\alpha_j(s) \ \square_{1,j} \ \ord\,t \ \square_{2,j} \ \beta_j(s) \ \wedge \\\ord\,t \equiv k_j \mod n \ \wedge \\\ac_m(t) = \xi_j\end{array}\right\}\right.,\]
where $\alpha_j, \beta_j: S_i \to \Gamma_K$ are definable functions and $(\square_j, k_j, n, m, \xi_j) \in P_K$.
To each $X_i$, we associate a definable set $\Sigma_i\subseteq S_i \times K^r$, which has the following property. To any function
\[ \sigma: S_i \to K^r : s \mapsto (\sigma_1(s), \ldots, \sigma_r(s)), \]
whose graph is contained in $\Sigma_i$, we can associate a (bijective) translation $T_\sigma: \sqcup_j C_j \to X_i$, defined by 
\[T_{\sigma}(s,t) = (s, t-\sigma_j(s)) \qquad \text{ for all } (s,t) \in C_j.\] 
The tuple $\{(X_i )_{i},(\Sigma_i)_{i}, (C_{\delta_{ij}})_{i,j}\}$,  will be called a $K$-cell decomposition of $X$.
%\begin{proof}
%Assume Theorem-Definition \ref{defthm:celdecB} holds (see the next section for this proof). Fix one of the sets $X_i$ in the partition of $X$. Then the set $\Sigma_i$ consists of all potential centers. Whenever we pick a function $\sigma$ with graph contained in $\Sigma_i$, we find that $X_i$ can be partitioned as
%%\begin{eqnarray*}X_i &=& \left\{(s,x) \in S_i \times K \ \left|  \bigvee_{j=1}^{r_i} C_{\delta_{ij}}(x,\sigma_j(s), \alpha_{ij}, \beta_{ij};s)\right\}\right.,\\
%%&=& 
%\[X_i = \bigcup_{j=1}^{r_i}\left\{(s,t) \in S_i \times K \ \left|\   C_{\delta_{j}}(t,\sigma_j(s), \alpha_{j}, \beta_{j};s)\right\}\right.\]%
%%\end{eqnarray*}
%It is then clear that the translation map $T_{\sigma}$ stated in the theorem gives the required bijection.
%%\textcolor{magenta}{
%%TO DO: DEDUCE THIS FROM PREVIOUS THEOREM}
%\end{proof}
%%For each such map $\sigma$, we have that
%\[(f\circ T_{\sigma})_{|C_j}(s,t) = \gamma_{j}(s) + a_j\frac{\ord(t) -k_j}{n},\]
%where $\gamma_j: S_i \to \Gamma_k$ is a definable function, $a_j \in \Z$.
\end{def-theorem} 
\noindent Notice that the cells $C_j$ are not necessarily disjoint (in fact, some of them may even coincide.) What we obtain is a family of bijective translations $T_{\sigma}$ between the disjoint union $\sqcup_j\, C_j$ and one of the parts $X_i$. Also note that while the sets $\Sigma_i$ are definable, we cannot assure that any of the individual curves $\sigma$ contained in it will be definable. (If a definable $\sigma$ exists and $X$ only consists of $K$-variables, a cell decomposition with the functions $\sigma(x)$ as centers will be very similar to what Mourgues obtained.)
%The $K$-cell decomposition theorem and this $\Gamma_K$-preparation result will be crucial in the proof of the integration and rationality results stated below. 
\\\\
In the second part of the paper (section \ref{sec:integration}), we discuss applications of the preparation and cell decomposition theorems. We will restrict our attention to the case where $K$ is a $p$-adic field. The results in other sections are valid for arbitrary $p$-adically closed fields.

 Inspired by his rationality results, Denef decided to introduce the class of constructible functions: 

\begin{defn}\label{def:constructible}
Let $X$ be an $\cL_2$-definable set. Write $\AA_{q_K}$ for the ring \[\AA_{q_K}:=\ZZ\left[q_K, q_K^{-1}, \left(\frac1{1-q_K^{-i}}\right)_{i\in \NN, i>0}\right].\]
We say that a function $f:X\to \QQ$ is \emph{$\cL_2$-constructible} if it is contained in the $\AA_{q_K}$-algebra generated by functions of the forms
%\begin{enumerate}
%\item $\alpha:X\to \ZZ$
%\item $X\to \ZZ:x\mapsto q_K^{\beta(x)}$,
%\end{enumerate}
\[\alpha:X\to \ZZ \qquad \text{and} \qquad
X\to \ZZ:x\mapsto q_K^{\beta(x)},\]
where $\alpha$ and $\beta$ are $\cL_2$-definable and $\ZZ$-valued.  
\end{defn}

When $\cL$ is $\cL_{\rm ring}$, the subanalytic language $\Lm_{an}$ on $K$ (see \cite{denef-vdd-88} for a definition),  or some intermediary languages as in \cite{CLip}, the class of $\Lm_2$-constructible functions is known to be stable under integration (see \cite{denef-2000}, \cite{clu-2003}, and \cite{Clu-Gor-Hal-14} for the most convenient dealing with integrability conditions). % These structures are the main examples of $P$-minimal structures, following the notion developed by Haskell and Macpherson \cite{has-mac-97}. 
%However, it is not known whether the analogous class of functions, build up from definable functions in a general $P$-minimal structure always has this stability property. 
We  show that (see Theorem \ref{thm:integrals}), whenever $(K, \Z)$ is a $P$-minimal structure, the class of $\Lm_2$-constructible functions is stable under integration : 

\begin{thm*}
Let $K$ be a $p$-adic field,  $(K, \Z)$  a $P$-minimal structure, and $f: X \subseteq S\times K^m \to \AA_{q_K}$ an  $\Lm_2$-constructible function such that $f(s, \cdot)$ is measurable and integrable on $Y_s$ for all $s\in S$. There exists a constructible function $g: S \to \AA_{q_K}$, such that
\[g(s) = \int_{X_s} f(s,x)|dx|,\]
for all $s\in S$.
\end{thm*}

We also extend the rationality results known so far only for the semi-algebraic \cite{denef-84} and subanalytic setting \cite{denef-vdd-88} (and thus also for any sublanguage), obtaining the following:

\begin{thm*}\label{thm:rationality}
Suppose that $(K,\Z)$ is $P$-minimal. Let $X$ be a definable subset of $\cO_K^n\times \NN$, and let $a_n$ be the Haar measure of $X_n:=\{x\in \cO_K^n\mid (x,n)\in X\}$ for each $n\geq 0$. Then the series $\sum_{i\geq 0} a_i T^i$ is rational. 
\end{thm*}
\noindent Here we normalize the Haar measure on $K^n$ so that $\cO_K^n$ has measure $1$. For a more precise statement we refer to  Corollary \ref{thm:rationalitygeneral}.  

\

%% file: preparation.tex
% !TEX root = main.tex
\section{Cell decomposition and function preparation} \label{sec:prep}

In this section we will give a proof of the cell decomposition and preparation theorems. For the comfort of the reader, we will restate (an abbreviated version of) the theorems. To ease notation, we will assume that for any definable set $X \subseteq S \times K$, the projection onto the parameter set $S$ is surjective, replacing $S$ by $\Pi_S(X)$ if necessary. The following notation will also be used in the proofs of this section.
\begin{defn}\label{def:cellcondition} A \emph{$K$-cell condition} is a formula $C_{\delta}(x,y,\alpha,\beta;s)$ of the form
 \begin{equation*}
C_\delta(x,y,\alpha,\beta;s):= \left(\begin{array}{l} \alpha(s)\ \square_1 \ \ord(x-y) \ \square_2 \ \beta(s) \ \wedge\\ \ord(x-y) \equiv k\mod n \ \wedge\\ \ac_{m}(x-y) = \xi \end{array} \right),
\end{equation*}
where $(\square,k,n,m,\xi)=\delta \in P_K$, and $\alpha, \beta$ are definable functions $S\to \Gamma_K$. When no $s$ appears, $\alpha,\beta$ are just elements of $\Gamma_K$. 
%\item A $\Gamma$-cell condition
\end{defn}

%It is easy to see that this can always be built into the definition of $S$ if necessary.

%In this section we work in a two sorted $P$-minimal $(K,\Gamma_K,\Lm)$. Although the applications in Section \ref{sec:integration} are restricted to the case where $K$ is a $p$-adic field, the results in this section hold for general $p$-adically closed fields. The following is a version of the well-known cell decomposition for semi-algebraic sets. 

%\begin{prop}[\cite{Clu-Gor-Hal-14}, Theorem 3.3.2]\label{prop.celldecomp}
%Let $X \subseteq  S \times K$ and $f_j:X \to\Gamma_K $ be $L_{ring,2}$-definable functions for $j=1,\ldots,r$. Then there exists a finite partition of $X$ into $p$-adic cells $C_i$ 
%\[C_i = \left\{(s,t) \in D\times K \ \left| \ \begin{array}{l} \alpha_i(s) \square_1 \ \ord(t-c_i(s)) \ \square_2 \ \beta_i(s),\\ \ord(t-c_i(s)) \equiv k_i\mod n_i,\\ \ac_{m_i}(t-c_i(s)) = \xi_i \end{array} \right\}\right.,\]
%such that for each occurring cell $C_i$ one has that for each $(s,t)\in C_i$
%\[f_j(s,t)=  h_{ij}(s)+ a_{ij}\frac{\ord(t- c_i(s))- k_i}{n_i},
%\]
%for $a_{ij}$ integers and $h_{ij}: D\to \Gamma_K$ $L_{ring,2}$-definable functions for $j=1,\ldots,r$. 
%\end{prop}
%
%An analogous statement is true where semi-algebraic is replaced by sub-analytic.  

%\textcolor{magenta}{THIS IS THEOREM-DEF 2.5; HOW ABOUT NUMBERING??}

\begin{theorem*}[Theorem-Definition \ref{thm:celdec-trans}]
Let $X \subseteq S \times K$ be a set definable in a $P$-minimal structure $(K, \Gamma_K)$. There exists a $K$-cell decomposition $\{(X_i )_{i},(\Sigma_i)_{i}, (C_{\delta_{ij}})_{i,j}\}$ of $X$. 
\end{theorem*}

\begin{proof}%[Proof of Theorem \ref{thm:celdec-trans}]

Fix a parameter $s \in S$. By the cell decomposition theorem for semi-algebraic sets, see e.g. \cite[theorem 3.3.2]{Clu-Gor-Hal-14}, there exists a finite partition of $X_s$ into $K$-cells 
\begin{equation}\label{eq:1varceldec}
C_s = \left\{t \in K \ \left| \ \begin{array}{l} \alpha_s\ \square_1 \ \ord(t-c_s) \ \square_2 \ \beta_s,\\ \ord(t-c_s) \equiv k_s\mod n_s,\\ \ac_{m_s}(t-c_s) = \xi_{m_s,s} \end{array} \right\}\right.,
\end{equation}
where $\alpha_s, \beta_s \in \Gamma_k, c_s \in K$ and $(\square,k_s, n_s, m_s,\xi_{m_s,s})\in P_K$. Note that the cell decomposition of $X_s$ may contain multiple cells of the same type. 

\begin{claim}\label{claim:compact} 
There is a natural number $N\geq 1$ such that for every $s\in S$, the set $X_s$ can be partioned as a union of at most  $N$ $K$-cells, and for each of these cells we can assure that $n_s,m_s<N$.
\end{claim}

The claim will follow by a standard compactness argument. Recall that $P_K$ consists of elements $\delta=(\square_\delta, k_\delta ,n_\delta, m_\delta, \xi_\delta)$, encoding the type of a $K$-cell. For each positive integer $N$, put 
\begin{align*}
P_{K,N}&:= \{\delta \in P_K \mid n_\delta <N, m_\delta <N\}, \intertext{ and write }
E_{K,N}&:= \bigsqcup_{i=1}^N P_{K,N},
\end{align*}
for the disjoint union of $N$ copies of $P_{K,N}$. Note that $E_{K,N}$ is a finite set. %For $\delta=(\square, k,n,m,\xi)\in P_{K,N}$ we let $k_{\delta}:=k$, $n_\delta:=n$, and so forth. 
For every $J\subseteq E_{K,N}$, fix an enumeration $\{\delta_1,\ldots,\delta_{|J|}\}$ of $J$. Given $y = (y_1,\ldots, y_{|J|})\in K^{|J|}$, and  $\alpha=(\alpha_1,\alpha_2)\in \Gamma_K^{2|J|}$  we will write
%\textcolor{blue}{
\[C_J(y,\alpha):=\left\{\begin{array}{ll}
\displaystyle{\bigcup_{i=1}^{|J|} C_{\delta_i}(K,y_i,\alpha_{1i},\alpha_{2i})} & \text{if the sets } C_{\delta_i}(K,y_i,\alpha_{1i},\alpha_{2i})\vspace{-8pt} \\ &\text{are disjoint,} \medskip 
%C_{\delta_i} \cap C_{\delta_j} = \emptyset, \text{for } i \neq j\\ 
\\\emptyset & \text{otherwise,} 
\end{array}\right.
\]
making use of $K$-cell conditions $C_{\delta_i}$ as defined in Definition \ref{def:cellcondition}.
Consider the set of formulas 
\[\Sigma(x):=\left.\left\{ \bigwedge_{J\subseteq E_{K,N}} \neg (\exists y\in K^{|J|})(\exists \alpha\in \Gamma_{K}^{2|J|})[X_x=C_J(y,\alpha)] \ \right|\ N\in\mathbb{N}^*\right\}.
\]
Since each $X_s$ can be partitioned in 
semi-algebraic cells as in  \eqref{eq:1varceldec}, $\Sigma(x)$ is inconsistent. Hence, by compactness there exists a finite subset $\Sigma_0(x)$ which is inconsistent. Since $\Sigma_0(x)$ is a finite subset of $\Sigma(x)$, one can find a positive integer $N_0$ such that 
\[\left[\bigwedge_{J\subseteq E_{K,N_0}} \neg (\exists y\in K^{|J|})(\exists \alpha\in \Gamma_{K}^{2|J|})[X_x=C_{J}(y,\alpha)] \right]\models \Sigma_0(x).
\end{equation*}
This implies that there must exist $N>0$ such that for every $s\in S$  
\[(K,\Gamma_K)\models \left[\bigvee_{J\subseteq E_{K,N}} (\exists y\in K^{|J|})(\exists \alpha\in \Gamma_{K}^{2|J|})[X_s=C_{J}(y,\alpha)] \right],
\]
which completes the claim.  

\

Now choose an integer $N$ satisfying the requirements of Claim \ref{claim:compact}. Let $W_N$ denote the power set of $E_{K,N}$. Since $W_N$ is finite, one can put a total ordering $\lessdot$ on it. We will also put an alternative ordering $\lhd$ on the value group $\Gamma_K$, which is defined by :
\begin{equation}\label{eq:order}
x \lhd y \Leftrightarrow
(0 \leqslant x <y) \vee (0< x \leqslant -y) \vee (0< -x < y) \vee (0< -x < -y).
\end{equation}
This produces a total ordering on $\Gamma_K$ which can be extended to $\Gamma_K^k$ lexicographically. We will also denote this extension by $\lhd$. The important property of the order $\lhd$ is that every definable set of $\Gamma_K$ has a $\lhd$-smallest element (for a proof of this, see the appendix, in particular \ref{preswellorder}). Now consider the map \[\tau: S \to W_N\times (\Gamma_K)^{\leq 2|W_N|}: s \mapsto (\tau_1(s), \tau_2(s)),\]where $\tau_1, \tau_2$ are defined as follows:
\begin{itemize}
\item put $\tau_1(s)= J$, if $J$ is the $\lessdot$-smallest element of $W_N$ such that  
$$(K,\Gamma_K)\models \left[(\exists y\in K^{|J|})(\exists \alpha\in \Gamma_{K}^{2|J|})[X_s=C_J(y,\alpha)]\right].$$
The claim ensures the existence of at least one such $J$ in $W_N$.    
\item let $\tau_2(s)$ be the $\lhd$-smallest tuple $\alpha\in \Gamma_K^{2|\tau_1(s)|}$ such that 
$$(K,\Gamma_K)\models (\exists y\in K^{|J|})[ X_s=C_{\tau_1(s)}(y,\alpha)].$$
\end{itemize}
It is clear that the function $\tau$ will be definable, using some fixed representation for the finite index set $W_N$. For each $J\in \tau_1(S)$, let $S_{J}$ be the set $\{s \in S \mid \tau_1(s)=J\}$. These sets induce a partition of $X$ into sets $X_J:=\{(s,x)\in X: s\in S_J\}$. We show that this partition satisfies all conditions stated in the theorem. Fix $\{\delta_1,\ldots,\delta_{|J|}\}=J\in \tau_1(S)$. The integer $r$ associated to $X_J$ is precisely $r:=|J|$. Let $\Sigma_J:=\bigcup_{s\in S_J} \Sigma_{J,s}$ be the set consisting of fibers
\[\Sigma_{J,s}:= \{y\in K^{r}\ |\ \left[X_s=C_{J}(y,\tau_2(s))\right]\}.\]
Note that these sets are non-empty by definition of $\tau_1$. Given any function $\sigma: S_J \to K^r: s \mapsto (\sigma_1(s), \ldots, \sigma_r(s))$  whose graph is contained in $\Sigma_J$, one then has that $X_s=C_{J}(\sigma(s),\tau_2(s))$ for all $s\in S_J$. For $1\leq j\leq r$ and $i \in \{1,2\}$, define $\alpha_{ij}: S_J\to \Gamma_K$ to be the $ij^{\text{th}}$-component in the tuple $\tau_2(s)$. We obtain that 
\[X_J=\bigcup_{j=1}^r \{(s,x)\in S_J\times K  \ \mid C_{\delta_j}(x,\sigma_j(s),\alpha_{1j},\alpha_{2j},s)\}.\]
Taking the $r$ $K$-cells associated to $X_J$ given by 
\[C_j:=\{(s,x)\in S_J\times K \ \mid \  C_{\delta_j}(x,0,\alpha_{1j},\alpha_{2j},s) \}, 1\leq j\leq r,\] 
it is clear that the translation map $T_{\sigma}$ stated in the theorem gives the required bijection.

%\[ C_j := \left\{(s,t) \in S_J \times K \left| \begin{array}{l}\alpha_1j(s) \ \square_{1,j} \ \ord\,t \ \square_{2,j} \ \alpha_2j(s) \ \wedge \\\ord\,t \equiv k_j \mod n \ \wedge \\\ac_m(t) = \xi_j\end{array}\right\}\right.,\]
%where $(\square_j, k_j, n, m, \xi_j)=\delta_j$. 
%
%
%
%
%
%%We can now write 
%%{\color{red} For $X_J:=\{(s,x)\in X: s\in S_J\}$ we have a finite partition $X=\bigcup\{X_J \ \mid \ J\in \sigma_1(S)\}$ where by construction
%%\small
%%\[X_J =\left\{(s,x)\in S_{J}\times K\left|\forall (c_1,\ldots,c_{|J|})\in \Sigma_{J,s}:\bigvee_{1\leq j\leq |J|}C_{\delta_j}(x,c_j,\alpha_{1j},\alpha_{2j};s)\right\}\right..
%\]
%\normalsize 
%Using the notation from Theorem-Definition \ref{defthm:celdecB}, we can then conclude that 
%%$\Sigma_J:=\bigcup_{s\in S_J} \Sigma_{J,s}$, 
% $\{(X_J)_J,(\Sigma_J)_{J},(C_{\delta_{iJ}})_{i,J}\}$ for $J\in W_N$, forms a $K$-cell decomposition for $X$. 
\end{proof}

We now present the preparation theorem for definable subsets of the form $X\subseteq S\times \Gamma_K$. The proof follows a similar scheme as the previous one.  

\begin{prop*}[Proposition \ref{prop:prep}]\label{prop:partialcd2} Let $(K, \Gamma_k, \Lm_2)$ be $P$-minimal. Let $X \subseteq S \times \Gamma_K$ and $f: X \to \Gamma_K$ a definable function. There exists a finite decomposition of $X$ into $\Gamma$-cells $C$ 
%\[C = \left\{(x,\gamma) \in D \times \Gamma_K \ \left| \ \begin{array}{l}\alpha(x)\ \square_1 \ \gamma \ \square_2 \  \beta(x)\quad \text{and}\quad \gamma \equiv n_0 \mod n \end{array} \right\}\right.,\]
such that on each such cell $C$, there exists a constant $a_{C} \in \ZZ$ and a definable function $\delta: D \to \Gamma_K$, such that for all $(x,\gamma) \in C$,
\[f(x,\gamma) = a_C \left(\frac{\gamma -n_0}{n}\right) + \delta(x).\]
 \end{prop*}
\begin{proof}
Since $(K, \Lm)$ is $P$-minimal, it follows from Theorem \ref{thm:semialgpres} that each of the fibers $X_s$ is Presburger definable. Cluckers \cite{clu-presb03} obtained a cell decomposition theorem for Presburger structures. Applying this to the sets $X_s$, yields that each
 $X_s$ can be partitoned into a finite union of cells of the form
\[C_s := \{ \gamma \in K \mid  \alpha_s\ \square_1 \ \gamma \ \square_2 \  \beta_s\quad \text{and}\quad \gamma \equiv n_0 \mod n_s\},\]
where $\alpha_s, \beta_s \in \Gamma_K$ and $n_s \in \NN$ are constants depending on $s$. Also note that for any $s \in S$, the graph of the function $f_s$ will be a Presburger set, by the assumption of $P$-minimality. Indeed, the related set 
\[G_s:=\{(x,y) \in (K^{\times})^2 \mid f_s(\ord\,x) = \ord\,y\}\]
is definable in a $P$-minimal structure, and hence by Theorem \ref{thm:semialgpres}, the set
\[Graph(f_s)=\{(\ord\,x,\ord\,y)\in \Gamma^2\mid (x,y)\in G_s\},\]
is Presburger definable. This means that each $f_s$ is a Presburger definable function, and hence  must be piecewise linear (with coefficients in $\QQ$). In particular, the above partition can be taken such that on each $C_s$, there exist constants $a_s \in \ZZ, \delta_s \in \gamma_K$ such that for all $\gamma \in C_s$, we have that
\[f_s(\gamma) = a_s \left(\frac{\gamma -n_0}{n_s}\right) + \delta_s.\]

\begin{claim}\label{cla:compact1} 
There is a natural number $N\geq 1$ such that for every $s\in S$, the set $X_s$ can be partioned as a union of at most $N$ $\Gamma$-cells, and for each of these cells we can assure that $n_s,|a_s|<N$.
\end{claim}

The claim follows by compactness and $P$-minimality using an analogous argument to the one presentend in Claim \ref{claim:compact}. 

 For an integer $N$ satisfying the requirements of the claim, let \begin{equation*}\label{defPK}
P_{\Gamma,N}:= \{(\square, k,n)\in P_\Gamma \mid n<N\}\hspace{0.5cm}\text{and}\hspace{0.5cm} E_{\Gamma,N} := \bigsqcup_{i=1}^N P_{\Gamma,N}.
\end{equation*}
Recall that $P_{\Gamma}$ consists of elements $\delta=(\square_\delta, k_\delta ,n_\delta)$, encoding the type of a $\Gamma$-cell. 
We will use the notation
\[ C_{\delta}(\alpha, \beta):= \left\{ \gamma \in K \mid \alpha \ \square_{\delta,1} \ \gamma \ \square_{\delta,2} \ \beta \wedge 
\gamma \equiv k_\delta \mod n_\delta \right\}.\]
%\sout{For $\delta=(\square, k,n)\in P_{\Gamma,N}$ we let $k_{\delta}:=k$ and $n_\delta:=n$.} 
For every $J\subseteq E_{\Gamma,N}$, fix an enumeration $\{\delta_1,\ldots,\delta_{|J|}\}$ of $J$.  %Given $y = (y_1,\ldots, y_{|J|})\in K^{|J|}$, and  $\alpha=(\alpha_1,\alpha_2)\in \Gamma_K^{2|J|}$  we will write
%Given $J\subseteq E_{\Gamma,N}$, $\{\delta_1,\ldots,\delta_{|J|}\}$ a fixed enumeration of $J$,
Given $\alpha=(\alpha_1,\alpha_2)\in \Gamma_K^{2|J|}$, we put
\[C_J(\alpha):=\left\{\begin{array}{ll}
\displaystyle{\bigcup_{i=1}^{|J|} C_{\delta_i}(\alpha_{1i},\alpha_{2i})} & \text{if the sets } C_{\delta_i}(\alpha_{1i},\alpha_{2i})\vspace{-8pt} \\ &\text{are disjoint,} \medskip 
%C_{\delta_i} \cap C_{\delta_j} = \emptyset, \text{for } i \neq j\\ 
\\\emptyset & \text{otherwise.} 
\end{array}\right.
\]
Let $a\in\ZZ^{|J|}$ be such that $|a_i|<N$ for all $1\leq i\leq |J|$.  Let $x$ be a tuple of variables of the same length (and sorts) as elements in $S$, and $\gamma$ a $\Gamma_K$-variable of length 1. The tuple $\alpha=(\alpha_1,\alpha_2,\alpha_3)$ consists of (tuples of) $\Gamma_K$-variables: $\alpha_1,\alpha_2$ and $\alpha_3$ all have length $|J|$. We define the formula $\phi_{J,a}(x,\alpha)$
as
\small
\begin{equation*}\label{cellformula}
\phi_{J,a}(x,\alpha):=\left(\begin{array}{l}
X_x=C_J(\alpha_1,\alpha_2)\ \wedge \\
\displaystyle\bigwedge_{1\leq i\leq |J|} (\forall \gamma) \left[ \gamma \in C_{\delta_i}(\alpha_{1i},\alpha_{2i})\to \left(f_x(\gamma)=a_i\left(\frac{\gamma-k_{\delta_i}}{n_{\delta_i}}\right)+\alpha_{3i}\right)\right]\end{array}\right).\\
\end{equation*}
\normalsize
Roughly, this formula states that the set $X_x$ can be decomposed into finitely many disjoint $\Gamma$-cells, on each of which the function $f_s$ satisfies the required preparation condition. Define the set $W_N$ by 
\begin{equation*}\label{WN}
W_N:=\{(J,a): J\subseteq E_{\Gamma,N}, a\in \ZZ^{|J|}, |a_i|<N \text{ for all } 1\leq i\leq |J|\}.
\end{equation*}
Since $W_N$ is finite, one can put a total ordering $\lessdot$ on it. As before we work with an alternative total ordering $\lhd$ on the value group $\Gamma_K$ defined as in equation (\ref{eq:order}). We proceed as in the $K$-cell decomposition theorem and define a map \[\sigma: S \to W_N \times (\Gamma_K)^{\leq 3|W_N|}: s \mapsto (\sigma_1(s), \sigma_2(s)),\]where $\sigma_1, \sigma_2$ are defined as follows:
\begin{itemize}
\item put $\sigma_1(s)= (J,a)$, if $(J,a)$ is the $\lessdot$-smallest element of $W_N$ such that  
$$(K,\Gamma_K)\models (\exists \alpha\in \Gamma_{K}^{3|J|})\phi_{J,a}(s,\alpha).$$
Claim \ref{cla:compact1} ensures the existence of at least one such $(J,a)$ in $W_N$.    
\item let $\sigma_2(s)$ be the $\lhd$-smallest tuple $\alpha\in \Gamma_K^{3|\sigma_1(s)|}$ such that 
$$(K,\Gamma_K)\models \phi_{\sigma_1(s)}(s,\alpha).$$
\end{itemize}
It is easy to see that the function $\sigma$ will be definable, using some fixed representation of the finite index set $W_N$. 
We recover the $\Gamma$-cell decomposition for $X$ and the linear functions satisfying the preparation condition in the following way. For each $\lambda=(J,a)\in \sigma_1(S)$, we define sets $S_\lambda$ and $X_{\lambda}$, as \[S_{\lambda}:=\{s\in S \ \mid \sigma_1(s)=\lambda\} \quad  \text{and} \quad X_\lambda:=\{(s,x)\in X: s\in S_\lambda\}.\] This gives us a finite partition of $X$ as $X = \cup_{\lambda} X_{\lambda}$. We will now partition the sets $X_{\lambda}$  as a finite union of $\Gamma$-cells, on each of which $f$ will have the required form.

 For $1\leq j\leq |J|$ and $i\in\{1,2,3\}$, define $\alpha_{\lambda ij}: S_\lambda\to \Gamma_K$ to be the $ij^{\text{th}}$-coordinate of $\sigma_2(s)$. Note that these functions are indeed definable, since $\sigma_2$ is definable. The above construction now implies that
 \[ X_{\lambda} = \bigcup_{1 \leqslant j\leqslant |J|} C_{\lambda, j},\]
 where 
 \(C_{\lambda, j} := \left\{(s, \gamma) \in S_\lambda \times K \mid  \alpha_{\lambda 1j} \ \square_{\delta_j,1} \ \gamma \ \square_{\delta_j,2} \ \alpha_{\lambda 2j} \wedge 
\gamma \equiv k_{\delta_j} \mod n_{\delta_j} \right\}.\) The formula $\phi_{\lambda}(s, \alpha)$ then ensures that for all $(s,\gamma) \in C_{\lambda, j}$, it holds that
\[f_{|C_{\lambda,j}}(s,\gamma) =a_j\left(\frac{t-k_{\delta_j}}{n_{\delta_j}}\right)+\alpha_{\lambda 3j},\]
which completes the proof.
 
%The partition $X=\bigcup\{X_\lambda \ \mid \ \lambda\in \sigma_1(S)\}$ satisfies by construction that for each $\lambda=(J,a)\in \sigma_1(S)$, each $1\leq j\leq |J|$, each $s$ such that $\sigma_1(s)=\lambda$ and each $\delta_j\in J$ ($1\leq j\leq|J|$) \small
%\[\forall t\in \Gamma_K \left(C_{\delta_j}(t,\alpha_{\lambda 1j},\alpha_{\lambda 2j};s) \rightarrow \left(f_s(t)=a_j\left(\frac{t-k_{\delta_j}}{n_{\delta_j}}\right)+\alpha_{\lambda 3j}\right)\right).
%\]}
%\normalsize
\end{proof}

\begin{remark}\label{rem} 
Proposition \ref{prop:prep} can be used to \emph{translate} theorems for parametrized Presburger definable sets $X\subseteq S\times \Gamma_K^m$ to two-sorted $P$-minimal structures, in the following sense.
The same theorems will hold in any two-sorted $P$-minimal structure, where the parameter set $S$ can now be any $\Lm_2$-definable set containing variables in both $K$ and $\Gamma_K$, and the involved Presburger-definable functions should be replaced by functions which are piecewise linear in the $\Gamma_K$-variables, in the sense of Proposition \ref{prop:prep}. 
The corollary stated below is an example of this. 

%Given a two-sorted $P$-minimal structures $(K,\Gamma_K,\Lm_2)$, Theorem \ref{prop:partialcd2} can be used to generalize theorems for parametrized Presburger definable sets $X\subseteq S\times \Gamma_K^m$ allowing the parameter set $S$ to be any $\Lm_2$-definable set containing variables in both $K$ and $\Gamma_K$. The corollary stated below is an example of this. 
\end{remark}

Given a $P$-minimal structure $(K,\Gamma_K,\Lm_2)$, we call a definable function $f : X \subseteq S\times\Gamma_K^m \to S\times\Gamma_K^l$ \emph{linear over $S$} if there is a definable function $g: S \to \Gamma_K^l$ and a linear definable function $a: \Gamma_K^m \to \Gamma_K^l$ such that $f(s,t)=(s,g(s)+ a(t))$ for all $(s,t)\in X$. We write $H$ for the set $H:=\{x\in \Gamma_K \mid x\geqslant 0\}$.

\begin{cor}(Parametric rectilinearization)\label{cor:recti}
Let $(K,\Gamma_K,\Lm_2)$ be a $P$-minimal structure and $X\subseteq S\times \Gamma_K^m$ be a definable set. There exists a finite partition of $X$ into
definable sets such that the following holds.

For each part $A$,  there is a set $B\subseteq  S\times \Gamma_K^m$ and a definable bijection  $\rho:A \to B$ which is linear over $S$ such that, for each $s\in S$, the set $B_s$ is
a set of the form $\Lambda_s\times H^l$ for a bounded subset $\Lambda_s\subseteq H^{m-l}$, depending on $s$ (in a definable way), and for an
integer $l\geq 0$ only depending on $A$.

\end{cor} 

%Given a $P$-minimal structure $(K,\Z,\Lm_2)$, we call a definable function $f : X \subseteq S\times\Z^m \to S\times\Z^l$ \emph{linear over $S$} if there is a definable function $g: S \to \Z^l$ and a linear definable function $a: \Z^m \to \Z^l$ such that $f(s,t)=(s,g(s)+ a(t))$ for all $(s,t)\in X$. We write $H$ for the set $H:=\{x\in \Gamma_K \mid x\geqslant 0\}$.
%
%\begin{cor}(Parametric rectilinearization)\label{cor:recti}
%Let $(K,\Z,\Lm_2)$ be a $P$-minimal structure and $X\subseteq S\times \Z^m$ be a definable set. There exists a finite partition of $X$ into
%definable sets such that the following holds.
%
%For each part $A$,  there is a set $B\subseteq  S\times \Z^m$ and a definable bijection  $\rho:A \to B$ which is linear over $S$ such that, for each $s\in S$, the set $B_s$ is
%a set of the form $\Lambda_s\times \N^l$ for a finite subset $\Lambda_s\subseteq \N^{m-l}$ depending on $s$ and for an
%integer $l\geq 0$ only depending on $A$.
%\end{cor}

\begin{proof} The proof is almost word for word the proof the same as the proof of the Parametric rectilinearization Theorem for Presburger definable sets (Theorem 3 in \cite{clu-presb03}). One just needs to replace every application of the Presburger function preparation theorem (Theorem 1 in \cite{clu-2003}) by Proposition \ref{prop:prep}.

%
%We provide the main idea. By Theorem \ref{prop:partialcd2} we may assume that $X$ is a $\Gamma$-cell 
%
%\[X= \left\{(s,x,t)\in D\times \Z \left|\begin{array}{l} \alpha(s,x)\ \square_1 \ \gamma \ \square_2 \ \beta(s,x), \\
%\gamma \equiv k\mod n \end{array}\right\}\right..\]
%By subsequently applying the induction hypothesis to the definable set $D$ and possibly partitioning further, we reduced to the case where $X$ has the form
%\[X = \{(s, x, t) \in D'\times \Z | \ 0 < t < \delta(s,x)\},\]
%where $\delta:D'\to \Z$ is a definable function and for each $s\in \Pi_S(D')$, $D_s' = \Lambda_s\times \N^l$ for $\Lambda_s$ a definable bounded set and $l$ a fixed positive integer. If $l = 0$, $X_s$ is a bounded set for each $s\in \Pi_S(X)$ and we are done. For the case $l\geq 1$, i.e., the projection
%of $X$ on the $x_m$-coordinate is $\N$, then (by Theorem \ref{prop:partialcd2}) the function $\delta$ can be written as $(s,x) \mapsto k_mx_m + \delta'(s,x_1,\ldots,x_{m-1})$ with $k_m$ an integer, necessarily nonnegative because the
%projection of $X$ on the $x_m$-coordinate is $\N$ and $\delta'$ is a linear function. The proof is finished by an induction on $k_m\geq 1$ exactly as in the proof of Theorem 3 in \cite{clu-presb03}. 
\end{proof}

%% file: integration.tex
% !TEX root = main.tex
\section{Integration and rationality} \label{sec:integration}

In this section, $K$ denotes a $p$-adic field, so the value group $\Gamma_K$ will just be $\Z$. Two types of integrals will appear. When integrating over (subsets of) $K^m$, the Haar measure $\mu$ is used. When integrating over $\Z^n$, we use the counting measure. The notation $\int_X|dx|$ will be used in both contexts, adapting the measure $|dx|$ to the sort of the variables involved. 

The results below are stated for an $\Lm_2$-constructible function $f: X \subseteq {S \times Y}\to \AA_{q_K}$, where both $S$ and $Y$ are definable sets and $S$ is considered a parameter set. Note that both $S$ and $Y$ may contain variables in both the $K$-sort and the $\Z$-sort, unless explicitly stated otherwise. Recall that the definition of constructible functions was given in Definition \ref{def:constructible}. %We will also consider the following sets: 
For such a fuction $f$, we define its \emph{locus of integrability} as the set 
\[\text{Int}(f,S):=\{s \in S \mid f(s, \cdot) \text{ is measurable and integrable on } Y_s \}.\] The main result of this section is the following theorem.

%\begin{defn}
%Let $f:X\subseteq S\times Y\to \AA_{q_K}$ be an $\Lm_2$-constructible function. Define the following sets:
%\begin{align*}
%Z(f)&:=\{x\in X \mid f(x)=0\}\\
%\text{Iva}(f,S)&:=\{s \in S \mid f(s, \cdot) \text{ is identically zero on } Y_s\}\\
%\text{Bdd}(f,S)&:= \{s \in S \mid f(s, \cdot) \text{ is bounded on } Y_s \} \\
%\text{Int}(f,S) &:=\{s \in S \mid f(s, \cdot) \text{ is measurable and integrable on } Y_s \}
%\end{align*}
%\end{defn}
\begin{thm} \label{thm:integrals} Let $K$ be a $p$-adic field and $(K, \Z, \Lm_2)$ be a $P$-minimal structure.
Let $S$ be a definable set, and $f:X\subseteq S\times Y\to \AA_{q_K}$ an $\Lm_2$-constructible function such that $\text{Int}(f,S)=S$.  %Define the following sets:
%\begin{align*}
%\text{Iva}(f,S)&:=\{s \in S \mid f(s, \cdot) \text{ is identically zero on } Y_s\}\\
%\text{Bdd}(f,S)&:= \{s \in S \mid f(s, \cdot) \text{ is bounded on } T \} \\
%\text{Int}(f,S) &:=\{s \in S \mid f(s, \cdot) \text{ is measurable and integrable on } Y_s \}
%\end{align*}
%\begin{enumerate}
%\item \textcolor{red}{There exist $\Lm_2$-constructible functions $h_i: S \to \AA_{q_K}$ for $i=1,2,3$, such that 
%\[\text{Iva}(f,S) = Z(h_1), \quad \text{Bdd}(f,S) = Z(h_2), \text{\quad and \quad }
%\text{Int}(f,S) = Z(h_3).\]}
%\text{Bdd}(f,S) &= Z(h_2),\\
%.
%\end{align*}
%\item There exists an $\Lm_2$-constructible function $\tilde{f}: X \to \AA_{q_K}$, Presburger-constructible over $S$, such that $\Int(\tilde{f},S) = S$ and $f(s,y) = \tilde{f}(s,y)$ whenever $s \in \text{Int}(f,S)$.
There exists an $\Lm_2$-constructible function $g: S \to \AA_{q_K}$, such that \[g(s) = \int_{X_s}f(s,y)|dy|,\]
for all $s\in S$.
%\end{enumerate}
\end{thm}

This is a  partial generalization of results which were already proven for specific cases by Cluckers, Gordon and Halupczok in \cite{Clu-Gor-Hal-14}. The generalization is partial because their results do not require the assumption $\text{Int}(f,S)=S$,  instead relying on an interpolation lemma, replacing $f$ by a function $\tilde{f}$ that coincides with $f$ on its locus of integrability, and for which $\text{Int}(\tilde{f},S)=S$. If a similar interpolation lemma can be proven to hold in general $P$-minimal structures, the assumption that $\text{Int}(f,S)=S$, can be removed from our result as well.

%We do not know if this assumption is necessary. At the end of the section we show that under an interpolation lemma, the assumption can be removed. We leave as a conjecture the truth of such interpolation for $P$-minimal structures.}
\begin{prop}\label{prop:int1}
Theorem \ref{thm:integrals} holds when $X \subseteq \Z^r$.
\end{prop}
\begin{proof}
Note that in this case, $X$ is $\Lpres$-definable, by Theorem \ref{thm:semialgpres}. Proofs can be found in \cite{Clu-Gor-Hal-14}, Theorem 2.1.6. 
\end{proof}

\begin{prop}%[Theorems 3.1.3, 3.1.5 and 3.1.1 of \cite{Clu-Gor-Hal-14}]
Theorem \ref{thm:integrals} holds when $\Lm_2 = \Lringtwo$ or $\Lm_{\text{an},2}$.
\end{prop}
\begin{proof}
See \cite{Clu-Gor-Hal-14}, Theorem 3.1.1.
\end{proof}

As a first step towards a general proof of Theorem \ref{thm:integrals}, we show that it already holds when $Y\subseteq \Z^r$:

\begin{prop}\label{prop:int2}
Theorem \ref{thm:integrals} holds when $Y \subseteq \Z^r$. 
\end{prop}
\begin{proof}
This is essentially a consequence of Proposition \ref{prop:prep} (see also the remark on page \pageref{rem}). In \cite{Clu-Gor-Hal-14}, this proposition was proven under the assumption that $\Lm_2 = \Lringtwo$ or $\Lm_{\text{an},2}$. Part (1) corresponds to Theorem 3.4.5 and part (2) to Theorem 3.1.1 in \cite{Clu-Gor-Hal-14}. If one replaces their Parametric rectilinearization Theorem (Proposition 3.4.4 in \cite{Clu-Gor-Hal-14}) by Corollary \ref{cor:recti}, the same proof also works for two-sorted $P$-minimal structures. 
%The last part, in the case where $\text{Int}(f,S) =S$, is a reformulation of Theorem-Definition 4.5.1 in \cite{clu-loe-08}. %The original theorem is for the case where $S$ is Presburger-definable, but the same proof applies in this context. 
%The general case follows then from part two of the theorem.
%The original versions of these proofs assume that $S$ is either $\Lpres$-, $\Lring$- or $\Lm_{\text{an}}$-definable, but it is easy to see that the same proofs hold independent of the language used.
\end{proof}

We will reduce the general case to Proposition \ref{prop:int2} using the following observation on the measure of definable sets.

\begin{prop}\label{prop:constructible measure} Let $(K,\Gamma_K, \Lm_2)$ be a $P$-minimal structure.
Let $X \subseteq S\times T$ be a definable set, where $T$ is $K$ or $\Z$. There exists a constructible function $g:S\to \AA_{q_K}$, such that $g$ uniformly measures the fibers $X_s$, that is,
\[g(s) = \int_{X_s}|dt|,\]
whenever $X_s$ has finite measure. Moreover, the set 
\[\tilde{S}:=\{s\in S \ | \ X_s \text{ has finite measure}\},\] 
is definable.
\end{prop}
\begin{proof}
When $T = \Z$, this is a consequence of Proposition \ref{prop:int2}, where $f$ is the constant function $f(x) =1$ for all $x \in X$. The fact that the set $\tilde{S}$ is definable follows from the cell-decomposition part of Proposition \ref{prop:prep}. Indeed, a $\Gamma$-cell $B\subseteq \Z$ (as in definition \ref{def:cell}) has finite measure if and only if both $\square_1$ and $\square_2$ are $<$ on such cell. This is a definable condition. 

Let us now  consider the case where $T = K$. By (the translation version of) the $K$-cell decomposition theorem (i.e., Theorem \ref{thm:celdec-trans}), we can partition $X$ in parts $X_i \subseteq S_i \times K$. On each of these parts, for any choice of a function $\sigma$ with image contained in $\Sigma$, we have that
\[ \int_{(X_i)_s} |dt| = \int_{T_{\sigma}(\sqcup (C_j)_s)} |dt|
= \sum_j \left[\int_{(C_j)_s} |dt|\right],
\]
and hence it suffices to compute the integral $\int_{(C_j)_s} |dt|$.
%Let us now compute the measure of a fiber $A_{s,\gamma}$. The $K$-cell decomposition provides a set $\Sigma \subseteq D \times K$, and cell conditions $C_{\delta_j}$.  For every choice of $(s,\gamma, c_1, \ldots, c_l) \in \Sigma$, we get a decomposition
%\[A_{s,\gamma} = \bigcup_{j=1}^l \{x\in K \mid C_{\delta_j}(x,c_j,\alpha_j, \beta_j;s,\gamma)\}.\]
Assume that $C_j$ is the zero-centered $K$ -cell 
\[C_{j}:=\left\{(s, t) \in S_i \times K \left| \begin{array}{l} 
\alpha_j(s) \ \square_{j,1}\ \ord \ t \ \square_{j,2}\ \beta_j(s),\\
\ord \ t \equiv k_j \mod N,\\ \ac_M(x) = \xi_j
\end{array}
\right\}\right..\]
Computing the measures of these cells, we get that
\begin{eqnarray*}
%\mu(A_{s,\gamma}) &=& \int_{A_{s,\gamma}}|dx|\\
%&=& \int_{\hat{A}_{s,\gamma}}|du|
%\\
\int_{({C}_{j})_{s}}|dt|&=& \sum_{\tau \in T_j}\mu\left(\xi_j \pi_K^{k_j + \tau N} (1 + \pi^{M}\cO_K)\right),\\
&=&|\xi|q_K^{-(k_j+M)}\sum_{\tau \in T_j} (q_K^{-N})^\tau,
\end{eqnarray*}
where $\pi_K$ is a uniformizing element for $K$ and $T_j$ is the set \[T_j:=\{\tau \in \Gamma_K \mid \alpha_j(s ) \ \Box_{j,1}\ k_j + \tau N\ \square_{j,2}\ \beta_j(s)\}.\]
It is easy to see that $(C_j)_s$ (and hence $X$) can only have finite measure if $\square_{j,1}$ denotes $<$ for $j = 1, \ldots, l$. Since this is a property of the cell, this is a definable condition.

We get the following results for this sum. 
%If possible, we will suppress the variables and just write $\alpha_m, \beta_m, \ldots$ to keep things readable. 
If we put $\tilde{\alpha}:= \lfloor\frac{\alpha -k}{N}\rfloor+1$, and $\tilde{\beta}:= \lceil\frac{\beta -k}{N}\rceil-1$ (clearly these are still definable functions), then we get that
\begin{equation}\label{eq:intcases} \sum_{\tau \in T_j} (q_K^{-N})^\tau = \left\{
\begin{array}{ll}
\frac{q_K^{-N\tilde{\alpha}_j}}{1-q_K^{-N}} & \text{if\ \ } \square_{j,2} = \emptyset,\\ %denotes \emph{no condition}
\frac{1}{1-q_K^{-N}}(q_K^{-N\tilde{\alpha}_j}-q_K^{-N\tilde{\beta}_j}) & \text{if \ \ } \square_{j,2} = <.
\end{array}\right.\end{equation}
In both cases we obtain an $\Lm_2$-constructible function. Hence, we can conclude that $\mu(X_s)$ is given by a constructible function as well.
\end{proof}

We can now complete the proof of the main theorem:
%\begin{thm}\label{thm:integration}
%Let $(K, \Z)$ be a relatively $P$-minimal structure, $S$ a definable set and $f: X \subseteq S\times K^m \to \AA_{q_K}$ a constructible function.  There exists a constructible function $g: S \to \AA_{q_K}$, such that
%\[g(s) = \int_{X_s} f(s,x)|dx|,\]
%whenever $s \in \text{Int}(f,S)$.
%\end{thm}

\begin{proof}[Proof of Theorem \ref{thm:integrals}]
Since $\text{Int}(f,S)=S$, by Fubini's theorem, the general result can be obtained by iteration and we may assume that either $Y= \Z$, or $Y=K$. The first case is already included in Proposition \ref{prop:int2}, so we only need to consider the case $Y = K$.

A general constructible function $f: X \subseteq S\times K \to \AA_{q_K}$ has the form
\[f(s,x) = \sum_{i=1}^r a_i q_K^{f_{i0}(s,x)} \prod_{j=1}^{r'}f_{ij}(s,x),\]
where the $f_{ij}$ are definable functions $X \to \ZZ$, and $a_i \in \AA_{q_K}.$
Now put $\gamma = (\gamma_{ij})_{i,j}$ and consider the set
\[G:= \{(s,\gamma, x)\in S\times \Z^{(r'+1)r}\times K \mid  \gamma_{ij} = f_{ij}(s,x)\},\]
which is a permutated version of the combined graphs of the functions generating $f$. To ease notations, we will sometimes consider $G$ as a subset of $D\times K$, where $D = \Pi_{S \times \Gamma_K^{(r'+1)r}}(G)$. 
%Put $R = r(r'+1)$. To ease the notation somewhat, we will sometimes renumber $\gamma =(\gamma_i)_{i=1\ldots R}$, where the correspondence is $\gamma_{ij}:= \gamma_{(i-1)(r'+1)+j+1}$.
%
%We will partition this set in cells, using the cell decomposition results from the previous section. %It is easy to see that an iteration of Theorem \ref{thm:partialprep} and Proposition \ref{prop:partialcd2} yields the following. 
%By Theorem \ref{thm:partialprep}, there exists a $K$-cell decomposition $\{(A_i)_i, (\Sigma_i)_i, (C_{\delta_ij})\}$ of $\text{GR}(f)$. That is, each part $A_i$ has the form
%\[A_i = \left\{(s,\gamma, x) \in D_i \times K \ \left| \ \forall (c_1, \ldots c_{r_i}) \in (\Sigma_i)_{s,\gamma}: \bigvee_{j=1}^{r_i} C_{\delta_{ij}}(x,c_j, \alpha_{ij}, \beta_{ij};s,\gamma)\right\}\right.,\] 
%for some definable subset $D_i \subseteq S \times \Gamma_K^R$. This partitioning can now be used to compute the given integral. 
%
Let $\mu$ denote the usual Haar measure. %(i.e., normalized such that $\mu(cO_K) =1$. ) 
The integral of $f_s$ can be written as a sum ranging over $\text{Im}(f_s)$:

\[ \int_{X_s} {f(s,x)}|dx|
 = \sum_{\delta \in \text{Im}(f_s)} \delta \cdot \mu\{x\in X_s \mid f_s(x) = \delta\},\]
 and this sum can be expressed in terms of the variables $\gamma$, to obtain a sum
\[ \sum_{\gamma \in D_s}  \left[\left(\sum_{i=1}^r a_i q_K^{\gamma_{i0}} \prod_{j=1}^{r'}\gamma_{ij}\right)\cdot \mu\left(\{x\in X_s \mid \bigwedge_{ij} f_{ij}(s,x) = \gamma_{ij}\}\right)\right]\] 
This reduces the integral to a sum
\begin{equation}\label{eq:int}
 \int_{X_s} {f(s,x)}|dx| = \sum_{\gamma \in D_s} \left(\sum_{i=1}^r a_i q_K^{\gamma_{i0}} \prod_{j=1}^{r'}\gamma_{ij}\right) \cdot \mu(G_{s,\gamma}).
\end{equation}
Applying Proposition \ref{prop:constructible measure}, we know that $\mu(G_{s,\gamma})$ is given by a constructible function, whenever $G_{s,\gamma}$ has finite measure. Since this is a definable condition, we may as well assume that the measure of $G_{s, \gamma}$ is finite for all $s \in S$ and $\gamma \in \text{Im}(f_s)$. Hence, we can conclude that 
\begin{equation}\label{eq:int}
 \int_{X_s} {f(s,x)}|dx| = \sum_{\gamma \in D_s} h(s,\gamma),
\end{equation}
for some constructible function $h: D \to \AA_{q_K}$. Noticing that 
\[
\sum_{\gamma \in D_s} h(s,\gamma)=\int_{D_s} {h(s,\gamma)}|d\gamma|, 
\]
the result follows by Proposition \ref{prop:int2} applied to the constructible function $h$.

\end{proof}

As a consequence of Theorem \ref{thm:integrals}, we obtain the following rationality result.

%\section{Rationality}\label{sec:rationality}
%In this section we will give a proof of Theorem \ref{thm:rationality}, which we restate here in a more general version, that gives more details about the exact form of the rational function. 
\begin{cor}\label{thm:rationalitygeneral}
Suppose that $(K,\Z)$ is $P$-minimal. Let $X$ be a definable subset of $\NN\times D$, where $D$ is a compact subset of $K^m$. % Write $a_n$ for the Haar measure of $X_n,$%:=\{x\in K^m\mid (x,n)\in X\}$, for each $n\geq 0$. 
Then the series $\sum_{n\geq 0} \mu(X_n) T^n$ is a rational function. More precisely, 
\[\sum_{n\geq 0} \mu(X_n) T^n = \frac{Q(T) }{\prod_{i=1}^r(1-q_K^{-m_i}T^N_i)},\]
for certain integers $m_i,r \in \NN$, $N_i >0$ and $Q(T) \in \AA_{q_K}[T]$. 
\end{cor}
\begin{proof}

Applying Theorem \ref{thm:integrals} to the set $X \subseteq \NN\times D$, one can find a constructible function $g: \NN \to \AA_{q_K}$, such that
\[g(n) = \int_{X_n}|dx|.\]
This function must have the form 
\[g(n) = \sum_{i=1}^r a_iq_K^{\alpha_i(n)}\prod_j\beta_{ij}(n),\]
where $a_i \in \AA_{q_K}$, and the functions $\alpha_i$ and $\beta_{ij}$ are Presburger-definable functions $\NN \to \ZZ$, and hence it is actually $\Lringtwo$-constructible. Our claim now follows from Denef's rationality results in the semi-algebraic case, for which we refer to eg. \cite{denef-84, denef-85}.\end{proof}

We finish by presenting as a conjecture a version of interpolation for $P$-minimal constructible functions: 

\begin{conjecture}(Interpolation)\label{cong:interpolation} Let $K$ be a $p$-adic field and $(K, \Z, \Lm_2)$ be a $P$-minimal structure. For every $\Lm_2$-constructible function $f:X\subseteq S\times Y\to \AA_{q_K}$ there exists an $\Lm_2$-constructible function $g:X\subseteq S\times Y\to \AA_{q_K}$  such that $\text{Int}(g,S)=S$ and $f(s, y)= g(s, y)$ whenever $s\in \text{Int}(f,S)$. 
\end{conjecture}

Assuming this conjecture, Theorem \ref{thm:integrals} implies the full generalization of the stability result in \cite{Clu-Gor-Hal-14}. 

\begin{corollary}\label{for:fullthm} Let $K$ be a $p$-adic field and $(K, \Z, \Lm_2)$ be a $P$-minimal structure and suppose that the interpolation conjecture is true. Let $S$ be a definable set, and $f:X\subseteq S\times Y\to \AA_{q_K}$ an $\Lm_2$-constructible function. There exists an $\Lm_2$-constructible function $g: S \to \AA_{q_K}$, such that \[g(s) = \int_{X_s}f(s,y)|dy|,\]
whenever $s\in \text{Int}(f,S)$.
\end{corollary}

%\end{defn}

%% file: wellordered.tex
% !TEX root = main.tex
\section{Definably well-ordered structures}\label{appendix}
%\textcolor{magenta}{ I changed some of the notations here to make it more coherent with the rest of the paper. I saved the original versions as comments, so this should be easy to reverse if you don't like the changes}
\begin{defn} \label{def:dwo}
A structure $(M, \Lm)$ is said to be \emph{definably well-ordered} if there exists a definable linear order $\lhd$ on $M$, such that every definable subset of $M$ has a $\lhd$-minimal element. 
\end{defn}
%\begin{defn}
%An $L$-structure $M$ is said to be \emph{definably well-ordered} if there exists a definable linear order $\lhd$ on $M$, such that every definable subset of $M$ has a $\lhd$-minimal element. 
%\end{defn}

\begin{lem}\label{lem.elem} Suppose that $(M, \Lm)$ is definably well-ordered, and that $\lhd$ is defined by an $\Lm(a)$-formula. Then every structure $(N, \Lm(a))$ which is elementarily equivalent to $(M, \Lm(a))$, is definably well-orderable. 
\end{lem}
%\begin{lem}\label{lem.elem} Suppose that $M$ is a definably well-ordered $L$-structure and let $\lhd$ be defined by an $L(a)$-formula. Then every elementarily equivalent $L(a)$-structure $N$ of $M$ is definably well-orderable. 
%\end{lem}

\begin{proof}
Let $\phi(x,y)$ be an $\Lm(a)$-formula with $length(x)=1$. By definition, 
\[M\models (\forall y )[(\exists x) \phi(x,y)\to (\exists x) (\phi(x,y)\wedge (\forall z) [\phi(z,y)\to x\unlhd z])].
\]
Since $N$ and $M$ are elementarily equivalent as $\Lm(a)$ structures, this implies that every definable subset of $N$ has a $\lhd$-minimal element.
\end{proof}
%\begin{proof}
%Let $\phi(x,y)$ be an $L(a)$-formula with $length(x)=1$. By definition 
%\[M\models \forall y (\exists x (\phi(x,y))\to \exists x (\phi(x,y)\wedge \forall z (\phi(z,y)\to x\unlhd z))).
%\]
%Since $N$ is elementarily equivalent to $M$ as $L(a)$ structures, this implies that every definable subset of $N$ has a $\lhd$-minimal element.
%\end{proof}
\noindent The previous lemma shows that being a definably well-ordered structure is a property of $Th(M,a)$, where $a$ is a tuple of parameters used in a formula defining a linear order that satisfies the requirements of Definition \ref{def:dwo}.  We say that a theory $T$ is definably well-orderable if it has some definably well-ordered model where the linear order is 0-definable. The following lemma shows the relation with cartesian powers:  

%\begin{lem}\label{lem.cartesian} The following are equivalent:
%\begin{enumerate}
%	\item $M$ is definably well-ordered; 
%	\item there is an $L(M)$-definable linear order $\lhd$ on $M^n$ such that any definable subset of $M^n$ has a $\lhd$-minimal element.
%\end{enumerate}
%\end{lem}
\begin{lem}\label{lem.cartesian} The following are equivalent:
\begin{enumerate}
	\item $(M, \Lm)$ is definably well-ordered; 
	\item There is an $\Lm(M)$-definable linear order $\lhd$ on $M^n$ such that any definable subset of $M^n$ has a $\lhd$-minimal element.
\end{enumerate}
\end{lem}

\begin{proof}
That (1) implies (2) follows by equipping $M^n$ with the lexicographic order induced by the definable linear order on $M$. For the converse, suppose that $n>1$ and pick an element  $a\in M^{n-1}$. Let $\phi(x,y)$ be a formula defining $x\lhd y$. Then $\phi(x_1,a;y_1,a)$ defines a well-order on $M$.  
\end{proof}

%\sout{In light of \ref{lem.cartesian}, given a definably well-ordered structure $M$ we denote by $\lhd$ a fixed definable linear order on $M$ such that definable sets in any cartesian power of $M$ have $\lhd$-minimal elements.} 
For a theory to have models which are definably well-ordered is a very strong property. As an example we show that such theories have definable choice, and thus eliminate imaginaries (for definitions we refer to \cite{vdd-98} and \cite{hodges-97}). 

\begin{prop}
An definably well-ordered structure $(M,\Lm)$ has definable choice. 
\end{prop}

\begin{proof}
Let $X\subseteq M^{m+n}$ be a definable set, and $\lhd$ a fixed definable linear order on $M$, such that definable sets in any cartesian power of $M$ have $\lhd$-minimal elements. Define $f:\Pi_{m}(X)\rightarrow M^n$ to be the function sending $x$ to the $\lhd$-least element in $X_x$. Clearly, if $X_x=X_y$ then $f(x)=f(y)$. 
\end{proof}

\begin{cor}
A definably well-ordered structure $(M, \Lm)$ has definable Skolem functions. 
\end{cor}

\begin{cor}
A definably well-ordered structure $(M, \Lm)$ has uniform elimination of imaginaries. 
\end{cor}
Notice that being a definably well-ordered structure is stronger than having definable choice. For instance, the real field has definable choice, yet by a result of Ramakrishnan in \cite{ramakrishnan-12} every definable order embeds in $(\RR^n,<_{\text{lex}})$. Therefore, no definable linear order has minimal elements for all definable subsets of the real line. Using this, one can show that no reduct of the real field is definably well-orderable. 

Even though definably well-ordered structures have strong model-theoretic properties, they are not always model-theoretically tame. For instance, the theory of arithmetic is definably well-orderable and yet model-theoretically wild. The main example of a tame well-orderable theory is Presburger Arithmetic. That this theory is well-orderable, is a consequence of the following proposition: 

\begin{prop}\label{preswellorder}
Let $\Lm$ be a language containing $\{\leq,-\}$ \sout{(as $\Lm_{Pres}$)}. Then $Th(\ZZ,\Lm)$ is definably well-orderable.
\end{prop}

\begin{proof}
Consider the following definable order

\[x\lhd y \Leftrightarrow 
\begin{cases}
0\leq x<y \\
0\leq x<-y \\
0\leq -x<y\\
0\leq -x<-y \\
\end{cases}
\]

On $\ZZ$ this defines the following well-order: 
$$0 \lhd 1 \lhd -1 \lhd 2 \lhd -2 \lhd \cdots  .$$

Because of lemma \ref{lem.elem}, this completes the proof. 
\end{proof}

Notice that the linear ordering defined in the previous proposition does not necessarily define  a well-order on every $\ZZ$-group $G$. However, it does define a linear order such that for any definable subset $A\subseteq G$, $A$ has a $\lhd$-minimal element. As a corollary we get a result from \cite{clu-presb03} 

\begin{cor}
Presburger arithmetic has elimination of imaginaries. 
\end{cor}